\documentclass[reqno]{amsart} 
\usepackage[headings]{fullpage}
\usepackage{amsmath} 
\usepackage{paralist}
\usepackage{amssymb} 
\usepackage{amsthm}
\usepackage{amsfonts}
\usepackage{mathtools}
\usepackage[centertags]{amsmath}
\usepackage{fullpage}
\usepackage{comment} 
\usepackage{bm}
\usepackage{thmtools}
\usepackage{thm-restate}
\usepackage{hyperref, float}
\hypersetup{colorlinks=true,citecolor=purple,linkcolor=purple}

\usepackage[backrefs]{amsrefs}

\usepackage{tikz-cd}
\usepackage{graphicx}
\usepackage{mathrsfs}
\usepackage[export]{adjustbox} 
\usepackage{xcolor}
\usepackage[margin=1.25in]{geometry}
\usepackage{amscd}
\usetikzlibrary{shapes}
\usetikzlibrary{arrows}
\usetikzlibrary{positioning}

\theoremstyle{definition}
\newtheorem{theorem}{Theorem}[section]

\newtheorem{lemma}[theorem]{Lemma}

\newtheorem{definition}[theorem]{Definition}

\theoremstyle{definition}
\newtheorem*{thm*}{Theorem}

\newtheorem{example}[theorem]{Example}

\newcommand{\N}{\mathbb{N}}
\newcommand{\Z}{\mathbb{Z}}
\newcommand{\Q}{\mathbb{Q}}

\renewcommand{\a}{\alpha}
\renewcommand{\b}{\beta}
\newcommand{\g}{\gamma}

\renewcommand{\O}{\Omega}

\newcommand{\paren}[1]{\left( #1 \right)}
\newcommand{\set}[1]{\left\{ #1 \right\}}

\newcommand{\abs}[1]{\left| #1 \right|}

\DeclareMathOperator{\dep}{depth}
\DeclareMathOperator{\Min}{Min}
\DeclareMathOperator{\Spec}{Spec}
\DeclareMathOperator{\Ass}{Ass}
\DeclareMathOperator{\height}{ht}

\DeclareMathOperator{\rAss}{rAss}

\newcommand{\union}{\cup}
\newcommand{\inter}{\cap}
\newcommand{\bigunion}{\bigcup}
\newcommand{\biginter}{\bigcap}

\newcommand{\sub}{\subseteq}
\newcommand{\subn}{\subsetneq}
\newcommand{\nsub}{\nsubseteq}
\newcommand{\iso}{\cong}

\newcommand{\todo}[1]{\PackageWarning{To Do}{#1}}
\newcommand{\note}[1]{\textcolor{red}{\textbf{#1}}\todo{#1}}
\newcommand{\hide}[1]{}

\begin{document}

\title[Extremely Noncatenary Noetherian UFDs]{Completions of Extremely Noncatenary Noetherian UFDs}
\author{Eli B. Dugan and S. Loepp}


\begin{abstract}
Let $T$ be a complete local ring. We present necessary and sufficient conditions for $T$ to be the completion of a local (Noetherian) unique factorization domain $A$ such that there exist height one prime ideals $\{J_k\}_{k = 1}^{\infty}$ of $A$ satisfying the following conditions:
(1) $J_k = J_{\ell}$ if and only if $k = \ell$, 
(2) there exist positive integers $n \neq m$ such that for each $k \in \N$, there are two saturated chains of prime ideals of $A$ of the form    
    $J_k \subn J^{(1)}_{k,2} \subn \cdots \subn J^{(1)}_{k,n - 1} \subn M$ and $J_k \subn J^{(2)}_{k,2} \subn \cdots \subn J^{(2)}_{k,m - 1} \subn M,$ where $M$ is the maximal ideal of $A$, and 
(3) the prime ideals from condition (2) satisfy $J^{(i)}_{k,a} = J^{(j)}_{\ell,b}$ if and only if $i = j$, $k = \ell$, and $a = b$. We also find sufficient conditions for $T$ to be the completion of a local (Noetherian) unique factorization domain $B$ such that $B/J$ is not catenary for all height one prime ideals $J$ of $B$.
\end{abstract}

\maketitle

\section{Introduction}

In commutative algebra, the structure of commutative rings is a main focus, and understanding the prime ideals of a commutative ring leads to a better understanding of the ring. The set of prime ideals of a ring $R$ is called the prime spectrum of $R$, and is denoted by $\Spec(R)$. Note that $\Spec(R)$ is a partially ordered set (poset) with respect to inclusion. 
With this view, it is natural to ask when a given poset is isomorphic to the prime spectrum of some commutative ring. 
%
In \cite{Lewis}, Lewis showed that for any finite poset $X$, there exists a commutative ring whose prime spectrum is poset isomorphic to $X$. As a consequence, a commutative ring can have an unusual prime ideal structure.  For example, recall that a ring is called catenary if, for every pair of prime ideals $P \subseteq Q$ of the ring, all saturated chains of prime ideals starting at $P$ and ending at $Q$ have the same length. A ring that is not catenary is said to be noncatenary. A consequence of Lewis' result is that commutative rings can be arbitrarily noncatenary. However, one cannot use the same reasoning to show that Noetherian rings can be noncatenary since a Noetherian ring having Krull dimension at least two must have infinitely many prime ideals.



For a time, it was thought that all Noetherian domains are catenary, but in 1956 Nagata constructed a family of noncatenary Noetherian domains \cite{Nagata}. Still, this left open just how noncatenary a Noetherian ring could be. Heitmann made a pivotal step towards answering this in 1979, when he showed in \cite{Heitmann_noncatenary} that every finite poset can be embedded into the prime spectrum of some Noetherian domain in a way that preserves saturated chains.

With the understanding that Noetherian rings could be noncatenary, attention turned to specific types of Noetherian domains. In 1980, Ogoma showed in \cite{Ogoma} the existence of a noncatenary Noetherian integrally closed domain. Then, the existence of a noncatenary Noetherian unique factorization (UFD) was shown by Heitmann in 1993 (see Theorem 10 in \cite{Heitmann_UFDs}). 
%
This result was furthered by a theorem of Avery et al. in 2019 \cite{SMALL_2017} that gives necessary and sufficient conditions for a complete local (Noetherian) ring to be the completion of a noncatenary local UFD. A consequence of the theorem is that, for integers $m$ and $n$ with $2 < m < n$, there exists a local (Noetherian) UFD with saturated chains of length $m$ and $n$ from the zero ideal to the maximal ideal. That is, their methods allowed great precision in mirroring particular noncatenary behavior in the prime spectrum of a Noetherian UFD.
%
This work was further generalized in \cite{Colbert_Loepp}, where it is shown that every finite poset can be embedded into the prime spectrum of a Noetherian UFD in a way that preserves saturated chains.

Following this, a natural place to turn is to investigate the extent to which infinite posets can be embedded into the prime spectra of Noetherian UFDs in a way that preserves saturated chains. In \cite{Bonat_UFDs}, the authors construct Noetherian UFDs that exhibit novel ``infinitely noncatenary" behavior. 
To further explain the results in \cite{Bonat_UFDs}, we introduce two definitions that capture  ``infinitely noncatenary" behavior in two different ways. We use $\N$ to denote the set of positive integers.

\begin{definition}\label{infinitely1noncatenary}
We say that a local ring $A$ with maximal ideal $M$ is {\em infinitely 1-noncatenary} if there exist height one prime ideals $\{J_k\}_{k = 1}^{\infty}$ of $A$ satisfying the following conditions:

\begin{enumerate}
    \item $J_k = J_{\ell}$ if and only if $k = \ell$, 
    \item there exist positive integers $n \neq m$ such that for each $k \in \N$, there are two saturated chains of prime ideals of $A$ 
    
    $$J_k \subn J^{(1)}_{k,2} \subn \cdots \subn J^{(1)}_{k,n - 1} \subn M$$ $$J_k \subn J^{(2)}_{k,2} \subn \cdots \subn J^{(2)}_{k,m - 1} \subn M, \mbox{ and}$$ \label{twochainsdifferentlengths}
    
    \item the prime ideals from condition (\ref{twochainsdifferentlengths}) satisfy $J^{(i)}_{k,a} = J^{(j)}_{\ell,b}$ if and only if $i = j$, $k = \ell$, and $a = b$.
\end{enumerate}
\end{definition}

If a local ring $A$ is infinitely 1-noncatenary, then there are infinitely many height one prime ideals $\{J_k\}_{k = 1}^{\infty}$ of $A$ for which there are saturated chains of prime ideals of different lengths, one of length $n$ and one of length $m$, starting at $J_k$ and ending at the maximal ideal of $A$. Moreover, all of these chains are disjoint except at the height one prime ideals (each $J_k$ is part of exactly two of the chains) and the maximal ideal.
Note that a consequence of $A$ being infinitely 1-noncatenary is that $A/J_k$ is noncatenary for every $k \in \N$.

\begin{definition}\label{everywhere1noncatenary}
We say that a ring $A$ is {\em everywhere 1-noncatenary} if $A/J$ is noncatenary for all height one prime ideals $J$ of $A$.
\end{definition}

In \cite{Bonat_UFDs}, the authors identify a set of sufficient conditions for a complete local ring to be the completion of an infinitely 1-noncatenary local UFD. As a consequence, they show that there exists an infinite family of infinitely 1-noncatenary local UFDs. Furthermore, they construct an example of an everywhere 1-noncatenary local UFD.

In this paper, we expand on the results in \cite{Bonat_UFDs} to increase our understanding of local UFDs that are, in a sense, infinitely noncatenary. 
In the first main result of this paper, Theorem \ref{thm: main infinitely}, we find necessary and sufficient conditions for a complete local ring to be the completion of an infinitely 1-noncatenary local UFD. What is perhaps most surprising about this result is that these are precisely the same conditions as those that were found in \cite{SMALL_2017} to be necessary and sufficient for a complete local ring to be the completion of a noncatenary local UFD. In our second main result, Theorem \ref{thm: main everywhere}, we find sufficient conditions for a complete local ring to be the completion of an everywhere 1-noncatenary local UFD. A consequence of this theorem is that, for $n \geq 4$, there exists an everywhere 1-noncatenary local UFD of dimension $n$ (see Example \ref{example}).


To prove both of our main results, we appeal to Lemma \ref{lem: noncatenary-proving machine}, in which we lay out sufficient conditions for there to exist two saturated chains of prime ideals of different lengths from a height one prime ideal of a local ring to its maximal ideal.
To prove Theorem \ref{thm: main infinitely}, we apply Lemma \ref{lem: noncatenary-proving machine} to an infinite family of height one prime ideals of our constructed UFD, while to prove Theorem \ref{thm: main everywhere}, we apply Lemma \ref{lem: noncatenary-proving machine} to every height one prime ideal.

In Section \ref{Preliminaries}, we present preliminary results and definitions.  In Section \ref{sec: infinitely} we state and prove our first main result, Theorem \ref{thm: main infinitely}, while in Section \ref{sec: everywhere}, we state and prove our second main result, Theorem \ref{thm: main everywhere}.

\medskip

\noindent {\bf Notation and Conventions}. All rings in this paper are assumed to be commutative with unity. When we say that a ring is quasi-local, we mean that it has exactly one maximal ideal and when we say that a ring is local, we mean that it is quasi-local and Noetherian. We use the notation $(A,M)$ to denote a quasi-local ring $A$ whose maximal ideal is $M$.
If $(A,M)$ is a local ring, then $\widehat{A}$ denotes the $M$-adic completion of $A$. 
We use $\Min(T)$ to denote the set of minimal prime ideals of a ring $T$.
Finally, if $P$ is a prime ideal of the ring $T$, then the coheight of $P$ is defined to be dim$(T/P)$.

\section{Preliminaries}\label{Preliminaries}

In this section, we present relevant results and definitions from the recent literature. We start by stating the main result from \cite{Heitmann_UFDs}, in which necessary and sufficient conditions are found for a complete local ring to be the completion of a local UFD.



\begin{theorem}[\cite{Heitmann_UFDs}, Theorem 1 and Theorem 8] \label{HeitmannUFDTheorem}
A complete local ring is the completion of a local UFD if and only if it is a field, a discrete valuation ring, or it has depth at least two and no element of its prime subring is a zerodivisor.
\end{theorem}

Note that by Theorem \ref{HeitmannUFDTheorem}, if a complete local ring $T$ has Krull dimension at least two, then $T$ is the completion of a local UFD if and only if $T$ has depth at least two and no element of its prime subring is a zerodivisor. The complete local rings we consider have Krull dimension at least three.  Thus, for one of our complete local rings to be the completion of a local UFD, it is necessary that it has depth at least two and satisfies the condition that no element of its prime subring is a zerodivisor. 

We next state two foundational results for much of our work. Both can be viewed as generalizations of the prime avoidance lemma.

\begin{lemma}[\cite{Heitmann_UFDs}, Lemma 2] \label{lem: heitmann lemma 2}
Let $(T, M)$ be a complete local ring, let $C \sub \Spec(T)$ be a countable set of nonmaximal prime ideals, and let $D \sub T$ be a countable set of elements. If $I \in \Spec(T)$ is contained in no single $P \in C$, then $I \nsub \bigunion \set{r + P \mid P \in C, r \in D}$.
\end{lemma}

\begin{lemma}[\cite{Heitmann_UFDs}, Lemma 3] \label{lem: heitmann lemma 3}
Let $(T, M)$ be a local ring, let $C \sub \Spec(T)$ be a set of prime ideals of $T$, and let $D \sub T$ be a set of elements of $T$. Suppose further that $\abs{C \times D} < \abs{T/M}$. If $I \in \Spec(T)$ is contained in no single $P \in C$, then $I \nsub \bigunion \set{r + P \mid P \in C, r \in D}$.
\end{lemma}

Another fundamental result on which we rely is the following proposition from \cite{Heitmann_locals}. It provides sufficient conditions for a subring of a given complete local ring to be Noetherian and to have its completion be isomorphic to the given complete local ring. Recall that, given a complete local ring $T$, we wish to find a subring $A$ of $T$ such that the completion of $A$ is $T$ and such that $A$ is a UFD that satisfies certain properties.  We use Theorem \ref{thm: completion-proving machine} to show that the completion of $A$ is $T$.

\begin{theorem}[\cite{Heitmann_locals}, Proposition 1] \label{thm: completion-proving machine}
Let $(R, M \inter R)$ be a quasi-local subring of a complete local ring $(T, M)$. If $R \to T/M^2$ is onto and $IT \inter R = I$ for every finitely generated ideal $I$ of $R$, then $R$ is Noetherian and $\widehat{R} \iso T$.
\end{theorem}

We now recall the definition of an $N$-subring of a complete local ring, first defined in \cite{Heitmann_UFDs}. $N$-subrings will play a crucial role in our constructions.

\begin{definition}[\cite{Heitmann_UFDs}] \label{def: N-subring}
For $(T, M)$ a complete local ring, an {\em $N$-subring of $T$} is a quasi-local UFD $(R, M \inter R)$ contained in $T$ that satisfies the following properties:
\begin{enumerate}
    \item $\abs{R} \leq \sup\set{\aleph_0, \abs{T/M}}$, with equality only if $T/M$ is countable,
    \item $Q \inter R = (0)$ for all $Q \in \Ass(T)$, and
    \item If $t \in T$ is regular and $P \in \Ass(T/tT)$, then $\height(P \inter R) \leq 1$.
\end{enumerate}
\end{definition}

Given a complete local ring $T$, we aim to construct a UFD that satisfies particular properties, including that its completion is $T$. One common approach in the literature to do this is to construct an ascending chain of $N$-subrings of $T$ whose union is the desired UFD. An important step in doing this is to start with an $N$-subring of $T$ and construct a strictly larger $N$-subring of $T$ that satisfies some desirable properties. The next definition, taken from \cite{Heitmann_UFDs}, gives examples of properties in which we are interested. Specifically, given an $N$-subring $R$, we are interested in constructing another $N$-subring $S$ such that $S$ contains $R$, prime elements of $R$ remain prime in $S$, and $S$ is countable or has the same cardinality as $R$.

\begin{definition}[\cite{Heitmann_UFDs}] \label{def: A-extension}
Let $(T, M)$ be a complete local ring and let $(R, M \inter R)$ be an $N$-subring of $T$. We say that an $N$-subring $(S, M \inter S)$ of $T$ is an {\em $A$-extension of $R$} if $R \sub S$, prime elements of $R$ are prime in $S$ and $\abs{S} \leq \sup\set{\aleph_0, \abs{R}}$.
\end{definition}

For our purposes, we make a small addition to Definition \ref{def: A-extension}. Specifically, we are interested in $A$-extensions whose intersection with elements of a particular set of prime ideals of $T$ is the zero ideal.

\begin{definition} \label{A_X-extension}
Let $(T, M)$ be a complete local ring, let $(R, M \inter R)$ be an $N$-subring of $T$, and let $X \sub \Spec(T)$ be such that $X_i \inter R = (0)$ for all $X_i \in X$. We say that an $N$-subring $(S, M \inter S)$ of $T$ is an {\em $A_X$-extension of $R$} if $S$ is an $A$-extension of $R$ and $X_i \inter S = (0)$ for all $X_i \in X$.
\end{definition}

As mentioned earlier, to construct our UFDs, we take the union of an ascending chain of $N$-subrings. To ensure that this union is a UFD, we use Lemma \ref{lem: heitmann lemma 6}, found in \cite{Heitmann_UFDs}. In addition, we use Lemma \ref{lem: heitmann lemma 6} at intermediate stages of our construction to ensure that a union of an ascending chain of $N$-subrings is itself an $N$-subring. To state the lemma more succinctly, we make use of the following definition.


\begin{definition}[\cite{Heitmann_UFDs}] \label{def: gamma}
For $\O$ a well-ordered set and $\a \in \O$, let $\g(\a) \coloneqq \sup\set{\b \in \O \mid \b < \a}$.
\end{definition}

\begin{lemma}[\cite{Heitmann_UFDs}, Lemma 6] \label{lem: heitmann lemma 6}
Let $(T, M)$ be a complete local ring and let $(R_0, M \inter R_0)$ be an $N$-subring of $T$. Let $\O$ be a well-ordered index set with least element 0 such that either $\O$ is countable or, for every $\a \in \O$, we have $\abs{\set{\b \in \O \mid \b < \a}} < \abs{T/M}$. Suppose further that $\set{R_\a}_{\a \in \O}$ is an ascending collection of rings such that if $\g(\a) = \a$ then $R_\a = \bigunion_{\b < \a} R_\b$, while if $\g(\a) < \a$ then $R_\a$ is an $A$-extension of $R_{\g(\a)}$. Then $(S, M \inter S) \coloneqq \bigunion_{\a \in \O} R_\a$ satisfies the conditions to be an $N$-subring of $T$ except for possibly the cardinality condition, but satisfies the inequality $\abs{S} \leq \sup\set{\aleph_0, \abs{R_0}, \abs{\O}}$. Moreover, elements which are prime in some $R_\a$ are prime in $S$.
\end{lemma}

The following lemma is a technical result that we use in our construction of everywhere 1-noncatenary local UFDs.

\begin{lemma}[\cite{Jensen}, Lemma 3.1] \label{lem: jensen}
Let $(T, M)$ be a complete local ring, let $(R, M \inter R)$ be an $N$-subring of $T$, and let $X$ be a finite set of nonmaximal prime ideals of $T$ such that $X_i \inter R = aR$ for every $X_i \in X$. Let $I$ be a finitely generated ideal of $R$ with $c \in IT \inter R$. Then there exists an $A$-extension $(S, M \inter S)$ of $R$ such that $\abs{S} = \abs{R}$, $X_i \inter S = aS$ for every $X_i \in X$, and $c \in IS$.
\end{lemma}

We next state a result from \cite{SMALL_2017} giving necessary and sufficient conditions for a complete local ring to be the completion of a noncatenary local UFD. 
This result is of particular importance to us as the UFDs that we construct are noncatenary.


\begin{theorem}[\cite{SMALL_2017}, Theorem 3.8] \label{thm: avery}
Let $(T, M)$ be a complete local ring. Then $T$ is the completion of a noncatenary local UFD if and only if the following conditions hold:
\begin{enumerate}
    \item No integer of $T$ is a zero divisor,
    \item $\dep(T) \geq 2$, and
    \item There exists $P \in \Min(T)$ such that $2 < \dim(T/P) < \dim(T)$.
\end{enumerate}
\end{theorem}

We now turn to results found in \cite{Bonat_UFDs}.
To state the lemmas more succinctly, we introduce the following definitions.

\begin{definition}
For $T$ a commutative ring, we define \[\rAss(T) \coloneqq \bigunion_{z \emph{ regular in } T} \Ass_T(T/zT).\] For $R$ a subring of $T$, we define \[\rAss^{(R)}(T) \coloneqq \bigunion_{\substack{z \emph{ regular in } T \\ z \in R}} \Ass_T(T/zT).\]
\end{definition}

\begin{lemma}[\cite{Bonat_UFDs}, Lemma 2.6] \label{lem: bonat 2.6}
Let $(T, M)$ be a complete local ring with $\dep(T) \geq 2$, let $(R, M \inter R)$ be a countable $N$-subring of $T$, and let $I$ be a nonmaximal prime ideal of $T$ with $I \inter R = (0)$. Let $\set{Q_j}_{j \in \N}$ be a countable set of prime ideals of $T$ such that for every $j$, $Q_j \nsub I$ and $Q_j \nsub P$ for every $P \in \Ass(T) \union \rAss(T)$. Then there exists a countable $A_{\set{I}}$-extension $(S, M \inter S)$ of $R$ that contains a generating set for each $Q_j$.
\end{lemma}

\begin{lemma}[\cite{Bonat_UFDs}, Lemma 2.7] \label{lem: bonat 2.7}
Let $(T, M)$ be a complete local ring and let $(R, M \inter R)$ be a countable $N$-subring of $T$. Let $Q$ be a prime ideal of $T$ such that $Q \nsub P$ for every $P \in \Ass(T) \union \rAss^{(R)}(T)$. Let $X \coloneqq \set{Q_1, Q_2, \ldots, Q_n}$ be a (possibly empty) set of prime ideals of $T$ such that $Q \nsub Q_j$ for all $j = 1, 2, \ldots, n$. Then there exists a height one prime ideal $I$ of $T$ such that $I \sub Q$ and $I \nsub P$ for every $P \in \Ass(T) \union \rAss^{(R)}(T)$. Moreover, if $X$ is not empty then $I \nsub Q_j$ for all $j = 1, 2, \ldots, n$.
\end{lemma}

\begin{lemma}[\cite{Bonat_UFDs}, Lemma 2.8] \label{lem: bonat 2.8}
Let $(T, M)$ be a complete local ring with $\dep(T) \geq 2$ and suppose $(R, M \inter R)$ is a countable $N$-subring of $T$. Let $P_1, \ldots, P_s$ be height one prime ideals of $T$ such that, for every $i = 1, 2, \ldots, s$ we have that $P_i \nsub P$ for every $P \in \Ass(T) \union \rAss^{(R)}(T)$. Let $X$ be a (possibly empty) finite set of prime ideals of $T$ such that $P_i \nsub Q$ for every $Q \in X$ and for every $i = 1, 2, \ldots, s$. Then, there exists $a \in \biginter_{i=1}^s P_i$ with $a \notin \bigunion_{Q \in X} Q$ such that $(S, M \inter S) \coloneqq R[a]_{(R[a] \inter M)}$ is an $A$-extension of $R$ with $P_i \inter S = aS$ for every $i = 1, 2, \ldots, s$.
\end{lemma}

\begin{lemma}[\cite{Bonat_UFDs}, Lemma 3.1] \label{lem: bonat 3.1}
Let $(T, M)$ be a local catenary ring with $\dep(T) \geq 2$ and let $P_0$ be a minimal prime ideal of T with $n \coloneqq \dim(T/P_0) \geq 3$. Then there are infinitely many prime ideals $Q$ of $T$ satisfying the conditions that $P_0$ is the only minimal prime ideal contained in $Q$, $\dim(T/Q) = 1$, and $Q \nsub P$ for every $P \in \Ass(T) \union \rAss(T)$.
\end{lemma}

The following result from \cite{Bonat_UFDs} is very relevant to us, as it provides sufficient conditions for identifying the existence of an infinitely 1-noncatenary Noetherian UFD whose completion is a specified ring $T$. 

\begin{theorem}[\cite{Bonat_UFDs}, Theorem 3.2] \label{thm: bonat inf}
Let $(T, M)$ be a complete local ring such that no integer of $T$ is a zero divisor of $T$ and such that $\dep(T) \geq 2$. Let $\set{P_{0,1}, \ldots, P_{0,s}}$ be the minimal prime ideals of T and suppose that for $i = 1, 2, \ldots, s$, we have $\dim(T/P_{0,i}) = n_i \geq 3$. Then there exists a local UFD $(A, M \inter A)$ such that $\widehat{A} \iso T$ and such that, for all $n \in \N$ and for all $i = 1, 2, \ldots, s$, there exist saturated chains of prime ideals $(0) \subn J_n \subn J_{2,n}^{(i)} \subn \cdots \subn J_{n_i - 1,n}^{(i)} \subn M \inter A$ of A satisfying $J_{a,b}^{(i)} = J_{c,d}^{(j)}$ if and only if $i = j$, $a = c$, and $b = d$, and $J_n = J_m$ if and only if $n = m$.
\end{theorem}

We end this section with a result from \cite{Bonat_UFDs} demonstrating the existence of an everywhere 1-noncatenary local UFD. In Theorem \ref{thm: main everywhere}, we identify a class of everywhere 1-noncatenary local UFDs. 


\begin{theorem}[\cite{Bonat_UFDs}, Theorem 4.7] \label{thm: bonat everywhere eg}
Let $T \coloneqq \Q[[x,y,z,w,t]]/((x) \inter (y,z))$. There exists an everywhere 1-noncatenary local UFD $A$ such that $\widehat{A} \iso T$.
\end{theorem}

\section{Infinitely 1-Noncatenary Local UFDs and Their Completions} \label{sec: infinitely}

 The main result of this section is Theorem \ref{thm: main infinitely}, in which we find necessary and sufficient conditions for a complete local ring to be the completion of an infinitely 1-noncatenary local UFD. Of particular interest is that these conditions are the same as those given in Theorem \ref{thm: avery} for when a ring is the completion of a noncatenary local UFD. As such, the important part of the proof is showing that given these known necessary conditions on a complete local ring $T$, it is in fact possible to construct a subring of $T$ with this more pathological prime ideal structure and such that its completion is $T$.

To do this, we closely follow the path laid out in the proof of Theorem \ref{thm: bonat inf} in \cite{Bonat_UFDs}. In the first part of this proof, the authors begin with a complete local ring $T$ and construct a particular subring $A$ of $T$ whose completion is $T$. Then, in the second part, they exploit facts about $A$ attained via the construction to verify that for infinitely many height one prime ideals $J$ of $A$, there are saturated chains of prime ideals of $A$ of different lengths starting at $J$ and ending at the maximal ideal of $A$. 
The ability to verify that these chains exist given appropriate conditions on $A$ and $J$ has general applicability (we use it again in Section \ref{sec: everywhere}), and so we break this argument out as a separate lemma and present it first as Lemma \ref{lem: noncatenary-proving machine}.

To prove Lemma \ref{lem: noncatenary-proving machine}, we begin with a complete local ring $(T,M)$ and a subring $A$ of $T$ where both $T$ and $A$ satisfy specific properties. In particular, $T$ has two saturated chains of prime ideals, each with a coheight one prime ideal that contains exactly one minimal prime ideal. Furthermore, $A$ contains generating sets for both of these coheight one prime ideals of $T$, and the intersection of the height one prime ideals from both chains with $A$ yield the same prime ideal of $A$, which we call $J$. We then construct two saturated chains of prime ideals in $A$ from $J$ to $M \inter A$, where the first is obtained by intersecting prime ideals from the first chain in $T$ with $A$, and the second is obtained by intersecting prime ideals from the second chain in $T$ with $A$. 

Before we state and prove Lemma \ref{lem: noncatenary-proving machine}, we present a preliminary result that we use in the proof of Lemma \ref{lem: noncatenary-proving machine}.

\begin{lemma} \label{lem: gen set dim}
Let $A$ be a local ring with $\widehat{A} = T$. Let $Q$ be a prime ideal of $T$ and suppose that $A$ contains a generating set for $Q$. Then \[\dim\paren{\frac{A}{Q \inter A}} = \dim\paren{\frac{T}{Q}}.\]
\end{lemma}

\begin{proof}
As $A$ contains a generating set for $Q$, we know that $Q = (Q \inter A)T$. Since $T$ is the completion of $A$, we have
$$\widehat{\frac{A}{Q \cap A}}\cong \frac{T}{(Q \cap A)T} = \frac{T}{Q}.$$
It follows that \[\dim\paren{\frac{A}{Q \inter A}} = \dim\paren{\frac{T}{(Q \inter A)T}} = \dim\paren{\frac{T}{Q}},\] and so the lemma holds.
\end{proof}

\begin{lemma} \label{lem: noncatenary-proving machine}
Let $(T, M)$ be a complete local ring satisfying the following conditions:
\begin{enumerate}
    \item There exist distinct $P^{(1)},P^{(2)} \in \Min(T)$ with $n_1 = \dim(T/P^{(1)}) \geq 3$ and $n_2 = \dim(T/P^{(2)}) \geq 3$,
    \item There exist $Q^{(1)},Q^{(2)} \in \Spec(T)$ with $\dim(T/Q^{(1)}) = \dim(T/Q^{(2)}) = 1$, and with $P^{(1)}$ the only minimal prime ideal of $T$ contained in $Q^{(1)}$ and $P^{(2)}$ the only minimal prime ideal of $T$ contained in $Q^{(2)}$,
    \item There exist $P^{(1)}_1,P^{(2)}_1 \in \Spec(T)$ such that $P^{(1)} \subn P^{(1)}_1 \subn Q^{(1)}$, $P^{(2)} \subn P^{(2)}_1 \subn Q^{(2)}$, and $\height(P^{(1)}_1/P^{(1)}) = \height(P^{(2)}_1/P^{(2)}) = 1$, 
    \item There is a local subring $(A, M \inter A)$ of $T$ that contains generating sets for $Q^{(1)}$ and $Q^{(2)}$,
    and is such that $\widehat{A} \iso T$,
    and
    \item There is $J \in \Spec(A) \setminus \Min(A)$
    such that $J = P^{(1)}_1 \inter A = P^{(2)}_1 \inter A$.
\end{enumerate}
Then there exist saturated chains of prime ideals of $A$, $J \subn J^{(1)}_2 \subn \cdots \subn J^{(1)}_{n_1 - 1} \subn M \inter A$ and $J \subn J^{(2)}_2 \subn \cdots \subn J^{(2)}_{n_2 - 1} \subn M \inter A$ satisfying $J^{(i)}_a = J^{(j)}_b$ if and only if $i = j$ and $a = b$. Moreover, $Q^{(1)} \cap A = J^{(1)}_{n_1 - 1}$ and $Q^{(2)} \cap A = J^{(2)}_{n_2 - 1}$.
\end{lemma}

\begin{proof}
First note that since $P_1^{(1)} \subseteq Q^{(1)}$ and $P^{(1)}$ is the only minimal prime ideal of $T$ contained in $Q^{(1)}$, we have that $P^{(1)}$ is the only minimal prime ideal of $T$ contained in $P_1^{(1)}$. Since $\height(P^{(1)}_1/P^{(1)})= 1 $, we have ht$(P_1^{(1)}) = 1$. It follows that ht$(J) = \mbox{ht}(P^{(1)} \cap A) \leq 1$ and since $J$ is not a minimal prime ideal of $A$, ht$(J) = 1$. Now, by Lemma \ref{lem: gen set dim}, the coheight of $Q^{(i)} \cap A$ is one for $i = 1,2.$


Suppose that $Q^{(1)} \inter A \sub J$. Then $J = P^{(1)}_1 \inter A \sub Q^{(1)} \inter A$, and so $Q^{(1)} \inter A = P^{(1)}_1 \inter A$. Since $A$ contains a generating set for $Q^{(1)}$, we have \[Q^{(1)} = (Q^{(1)} \inter A)T = (P^{(1)}_1 \inter A)T \sub P^{(1)}_1,\]  contradicting that $P^{(1)}_1 \neq Q^{(1)}$. Thus, $Q^{(1)} \inter A \nsub J$. Similarly, $Q^{(2)} \inter A \nsub J$.
Next, suppose that $Q^{(1)} \inter A \sub Q^{(2)} \inter A$. Then $(Q^{(1)} \inter A)T \sub (Q^{(2)} \inter A)T$, and so $Q^{(1)} \sub Q^{(2)}$, a contradiction. Thus, $Q^{(1)} \inter A \nsub Q^{(2)} \inter A$. Similarly, $Q^{(2)} \inter A \nsub Q^{(1)} \inter A$.

We now begin to build our desired chains. By the prime avoidance lemma, there exists $j^{(1)}_1 \in (Q^{(1)} \inter A) \setminus (J \union (Q^{(2)} \inter A))$. Let $P^{(1)}_2$ be a minimal prime ideal of $(j^{(1)}_1) + P^{(1)}_1$ contained in $Q^{(1)}$.  In the ring $T/P^{(1)}_1$, $(j^{(1)}_1 + P^{(1)}_1)/P^{(1)}_1$ is principal, so by the principal ideal theorem, $P^{(1)}_1 \subn P^{(1)}_2$ is saturated.
Let $J^{(1)}_2 \coloneqq P^{(1)}_2 \inter A$. We show that $J \subn J^{(1)}_2$ is saturated. Suppose for contradiction that there exists $J' \in \Spec(A)$ such that $J \subn J' \subn J^{(1)}_2$. Since $T$ is a faithfully flat extension of $A$, the going-down property holds, and so there exists $P', P'' \in \Spec(T)$ with $P'' \subn P' \subn P^{(1)}_2$, $P' \inter A = J'$, and $P'' \inter A = J$. Since $Q^{(1)}$ contains only one minimal prime ideal, $P^{(1)} \sub P''$. As $J$ is not a minimal prime ideal of $A$, we have $P^{(1)} \subn P''$. Hence $P^{(1)} \subn P'' \subn P' \subn P^{(1)}_2$. Since $T$ is a complete local ring, it is catenary. Now $P^{(1)} \subn P^{(1)}_1 \subn P^{(1)}_2$ is saturated, violating that $T$ is catenary. It follows that $J \subn J^{(1)}_2$ is saturated. We follow a similar process to find $J^{(2)}_2$, where $J \subn J^{(2)}_2$ is also saturated.


If for $i \in \{1,2\}$, we have $n_i = 3$, then $P^{(i)}_2 = Q^{(i)}$, and so $J^{(i)}_2 = P^{(i)}_2 \inter A = Q^{(i)} \inter A$. 
Thus, when $n_i = 3$, we have successfully constructed a saturated chain $J \subn J^{(i)}_2 \subn M \inter A$. 
We now demonstrate an inductive approach in the case where $n_i > 3$. For concreteness, suppose $i = 1$. 
Suppose that $t \geq 3$ and that for $j = 1,2, \ldots ,t-1$, $P^{(1)}_{j}$ is a prime ideal of $T$ such that $P^{(1)}_{t-1} \subn Q^{(1)}$ and $P^{(1)} \subn P^{(1)}_{1} \subn \cdots \subn P^{(1)}_{t-1}$ is saturated. Moreover, suppose that for $j = 1,2, \ldots ,t-1,$
$J^{(1)}_{j} = P^{(1)}_{j} \inter A$ and $ J \subn J^{(1)}_{2} \subn \cdots \subn J^{(1)}_{t-1}$ is saturated.
Since $P^{(1)}_{t-1} \subn Q^{(1)}$,  dim$(T/P^{(1)}_{t-1}) \geq 2$. If $J^{(1)}_{t-1} = P^{(1)}_{t-1} \inter A = Q^{(1)} \inter A$, then
$$\mbox{dim}\left(\frac{A}{Q^{(1)} \cap A}\right) = \mbox{dim}\left(\frac{A}{P^{(1)}_{t-1} \cap A}\right) = \mbox{dim}\left(\frac{T}{(P^{(1)}_{t-1} \cap A)T}\right) \geq \mbox{dim}\left(\frac{T}{P^{(1)}_{t-1}}\right) \geq 2,$$
a contradiction. It follows that $J^{(1)}_{t-1} = P^{(1)}_{t-1} \inter A \subn Q^{(1)} \inter A$.
Choose $j^{(1)}_{t-1} \in Q^{(1)} \cap A$ such that $j^{(1)}_{t-1} \not\in J^{(1)}_{t-1}$, and choose $P^{(1)}_t$ to be a minimal prime ideal of $(j^{(1)}_{t-1}) + P^{(1)}_{t-1}$ contained in $Q^{(1)}$. Then $P^{(1)}_{t-1} \subn P^{(1)}_t$ is saturated. We set $J^{(1)}_t \coloneqq P^{(1)}_t \inter A$.
Now suppose that $J^{(1)}_{t-1} \subn J^{(1)}_t$ is not saturated. Then there is a prime ideal $J'$ of $A$ such that $J^{(1)}_{t-1} \subn J' \subn J^{(1)}_t$. By the going-down property there are prime ideals $P_1, P_2, \ldots, P_{t - 1}$, and $P'$ of $T$ satisfying $P_1 \subseteq P_2 \subseteq \cdots \subseteq P_{t-1} \subseteq P' \subseteq P^{(1)}_t$, $P_1 \cap A = J$, $P' \cap A = J'$, and $P_j \cap A = J_j^{(1)}$ for $j = 2,3, \ldots, t-1$. Note that $P^{(1)} \subn P_1$, and so we have $$P^{(1)} \subn P_1 \subn P_2 \subn \cdots \subn P_{t-1} \subn P' \subn P^{(1)}_t.$$ Now $$P^{(1)} \subn P^{(1)}_{1} \subn \cdots \subn P^{(1)}_{t-1} \subn P^{(1)}_{t}$$ is a saturated chain of prime ideals of $T$, violating that $T$ is catenary.  It follows that $J^{(1)}_{t-1} \subn J^{(1)}_t$ is saturated.


As we continue this inductive process, there will be some $\ell \in \N$ for which $P^{(1)}_{\ell-1} \subn Q^{(1)}$ and $P^{(1)}_\ell = Q^{(1)}$, at which point we stop. By construction, the chain $P^{(1)}_1 \subn P^{(1)}_2 \subn \cdots \subn P^{(1)}_\ell \subn M$ is saturated. As $T$ is catenary, $\ell = n_1 - 1$. Also, by construction, we conclude that $J \subn J^{(1)}_2 \subn \cdots \subn J^{(1)}_{n_1 - 1} \subn M \inter A$ is saturated. The same argument applies for $i = 2$. 

We now verify that no prime ideal appears in both chains, apart from those at the endpoints. Suppose for contradiction that $J^{(1)}_a = J^{(2)}_b$, for $2 \leq a \leq n_1 - 1$ and $2 \leq b \leq n_2 - 1$. Then \[j^{(1)}_1 \in J^{(1)}_2 \sub J^{(1)}_a = J^{(2)}_b \sub Q^{(2)},\] which means $j^{(1)}_1 \in Q^{(2)} \inter A$, but this contradicts our choice of $j^{(1)}_1$. By construction, $J^{(i)}_a \neq J^{(i)}_b$ for all $a \neq b$, so we have shown that $J^{(i)}_a = J^{(j)}_b$ if and only if $i = j$ and $a = b$. Also by construction, $Q^{(1)} \cap A = J^{(1)}_{n_1 - 1}$ and $Q^{(2)} \cap A = J^{(2)}_{n_2 - 1}$.
\end{proof}

We are now in a position to prove our first main result, Theorem \ref{thm: main infinitely}, in which we find necessary and sufficient conditions for a complete local ring to be the completion of an infinitely 1-noncatenary local UFD. That the conditions are necessary is immediate from Theorem \ref{thm: avery}, so the bulk of the proof entails following the argument given in the first half of the proof of Theorem \ref{thm: bonat inf} in \cite{Bonat_UFDs} to show that the conditions are sufficient. Note that, to use Theorem \ref{thm: bonat inf}, the minimal prime ideals of the complete local ring $T$ must all have coheight at least three. As our complete local ring $T$ does not satisfy that condition, we cannot simply apply Theorem \ref{thm: bonat inf}.

The idea of the proof of Theorem \ref{thm: main infinitely} is to set up a situation where we can apply Lemma \ref{lem: noncatenary-proving machine} infinitely many times.
 Given two minimal prime ideals of the complete local ring $(T,M)$, we first find infinitely many coheight one prime ideals of $T$ and then we find infinitely many height one prime ideals of $T$. We do this in a way so that we have pairs of coheight one prime ideals along with pairs of height one prime ideals that satisfy the first three conditions of Lemma \ref{lem: noncatenary-proving machine}.

It then remains to find our subring $A$ for the application of Lemma \ref{lem: noncatenary-proving machine}. To do this, we first construct an $N$-subring that contains a generating set for each of the coheight one prime ideals. Second, we construct an $A$-extension of that $N$-subring that ensures that the first pair of height one prime ideals intersects to the same prime ideal in the subring, and then we construct an $A$-extension of that $N$-subring that ensures the same property for the second pair of height one prime ideals, and so on. Third, we argue that the union of these $N$-subrings is again an $N$-subring. Finally, we construct $A$ as an extension of this infinite union, which can be shown to be a local UFD whose completion is $T$. The proof of the theorem is completed by appealing to Lemma \ref{lem: noncatenary-proving machine} infinitely many times.

\begin{theorem} \label{thm: main infinitely}
Let $(T, M)$ be a complete local ring. Then $T$ is the completion of an infinitely 1-noncatenary local UFD $(A, M \inter A)$ if and only if the following conditions hold:
\begin{enumerate}
    \item No integer of $T$ is a zero divisor,
    \item $\dep(T) \geq 2$, and
    \item There exists $P \in \Min(T)$ such that $2 < \dim(T/P) < \dim(T)$.
\end{enumerate}
\end{theorem}

\begin{proof}
Let $(T, M)$ be a complete local ring that is the completion of an infinitely 1-noncatenary local UFD $A$. Then $A$ is noncatenary, and so by Theorem \ref{thm: avery} the three conditions hold.

Now assume that $(T, M)$ is a complete local ring that satisfies our three conditions. The third condition guarantees that we have two minimal prime ideals $P^{(1)}$ and $P^{(2)}$ of $T$ satisfying dim$(T/P^{(1)}) \geq 3$, dim$(T/P^{(2)}) \geq 3$, and dim$(T/P^{(1)}) \neq \mbox{dim}(T/P^{(2)})$.

We now adapt the proof of Theorem \ref{thm: bonat inf}. By applying Lemma \ref{lem: bonat 3.1} to $P^{(1)}$, we find infinitely many coheight one prime ideals $\{Q_n^{(1)}\}_{n \in \N}$ of $T$ that contain $P^{(1)}$ (and which contain no other minimal prime ideals of $T$) such that, for each $n \in \N$, $Q^{(1)}_n \nsub P$ for all $P \in \Ass(T) \union \rAss(T)$. We find a similar collection $\{Q_n^{(2)}\}_{n \in \N}$ of coheight one prime ideals of $T$ containing $P^{(2)}$.

Let $\Pi$ be the prime subring of $T$, and set $R \coloneqq \Pi_{(M \inter \Pi)}$. Then $R$ is a countable $N$-subring of $T$. By Lemma \ref{lem: bonat 2.6} with $I \in \Ass(T)$, there exists a countable $A$-extension $R_0$ of $R$ such that, for every $n \in \N$, $R_0$ contains a generating set for $Q_n^{(1)}$ and for $Q_n^{(2)}$.

For $i \in \set{1,2}$, apply Lemma \ref{lem: bonat 2.7} with $Q = Q_1^{(i)}$ and $X$ the empty set to identify a height one prime ideal $P_1^{(i)}$ of $T$ such that $P^{(i)} \sub P_1^{(i)} \sub Q_1^{(i)}$ and $P_1^{(i)} \nsub P$ for all $P \in \Ass(T) \union \rAss^{(R_0)}(T)$. Then, by Lemma \ref{lem: bonat 2.8} (with $X$ empty) there exists $a_1 \in P_1^{(1)} \inter P_1^{(2)}$ such that $R_1 \coloneqq R_0[a_1]_{(M \inter R_0[a_1])}$ is an $A$-extension of $R_0$ satisfying $P_1^{(1)} \inter R_1 = a_1 R_1$ and $P_1^{(2)} \inter R_1 = a_1 R_1$. Note that $a_1$ is a prime element of $R_1$ and since $R_0$ is countable, $R_1$ is countable.

We now construct $R_n$ inductively for $n > 1$. Suppose that for $k < n$, $R_k$ is a countable N-subring of $T$, $P_k^{(1)}$ and $P_{k}^{(2)}$ are height one prime ideals of $T$ such that, for $i \in \set{1,2}$, we have $P^{(i)} \sub P_{k}^{(i)} \sub Q_{k}^{(i)}$, $P_{k}^{(i)} \nsub Q_j^{(1)}$ for all $j < k$, $P_{k}^{(i)} \nsub Q_j^{(2)}$ for all $j < k$, and $P_{k}^{(i)} \nsub P$ for all $P \in \Ass(T) \union \rAss^{(R_k)}(T)$. Suppose also that for $k < n$ there exists $a_{k} \in P_{k}^{(1)} \inter P_{k}^{(2)}$ satisfying $a_{k} \notin Q_j^{(1)}$ for all $j < k$ and $a_{k} \notin Q_j^{(2)}$ for all $j < k$. Moreoever, suppose that $P_{k}^{(1)} \inter R_{k} = a_{k} R_{k}$, and $P_{k}^{(2)} \inter R_{k} = a_{k} R_{k}$. Note that $a_k$ is a prime element of $R_k$.

Here, we make the inductive step. For each $i \in \set{1,2}$, apply Lemma \ref{lem: bonat 2.7} with $Q = Q_n^{(i)}$ and $X = \{Q_j^{(i)} \mid i \in \set{1,2}, 1 \leq j < n\}$ to find a height one prime ideal $P_n^{(i)}$ of $T$ such that $P^{(i)} \sub P_n^{(i)} \sub Q_n^{(i)}$, and $P_n^{(i)} \nsub P$ for all $P \in \Ass(T) \union \rAss^{(R_{n-1})}(T)$. Furthermore, the lemma guarantees that $P_n^{(i)} \nsub Q_j^{(1)}$ for all $j < n$ and $P_n^{(i)} \nsub Q_j^{(2)}$ for all $j < n$. We next use Lemma \ref{lem: bonat 2.8} with $X = \{Q_j^{(i)} \mid i \in \set{1,2}, 1 \leq j < n\}$, to find $a_n \in P_n^{(1)} \inter P_n^{(2)}$ with $a_n \notin \bigunion_{Q \in X} Q$ such that $R_n \coloneqq R_{n-1}[a_n]_{(M \inter R_{n-1}[a_n])}$ is an $A$-extension of $R_{n-1}$ that satisfies $P_n^{(1)} \inter R_n = a_n R_n$ and $P_n^{(2)} \inter R_n = a_n R_n$. Note that $a_n$ is a prime element of $R_n$ and $R_n$ is countable.

Let $S \coloneqq \bigunion_{n=1}^\infty R_n$. By Lemma \ref{lem: heitmann lemma 6}, $S$ is a countable $N$-subring of $T$, and $a_n$ is a prime element of $S$ for all $n \in \N$. Now follow the proof of Theorem 8 from \cite{Heitmann_UFDs}, replacing $R_0$ in that argument with $S$, to construct a local UFD $(A, M \inter A)$ whose completion is $T$. In particular, $S \sub A$, and prime elements of $S$ are prime in $A$. As $S$ contains a generating set for $Q_n^{(i)}$ for all $n \in \N$, and for $i \in \set{1,2}$, $A$ will as well. Also, for every $n \in \N$, $a_n$ is a prime element of $A$ and so $P_n^{(1)} \inter A = a_n A$, and $P_n^{(2)} \inter A  = a_n A$. Let $J_n \coloneqq a_n A$, and note that $J_n = J_m$ if and only if $n = m$, as $a_n$ satisfies $a_n \not\in Q_j^{(1)}$ and $a_n \not\in Q_j^{(2)}$ for all $j < n$.
The result now follows from Lemma \ref{lem: noncatenary-proving machine}.
\end{proof}

\section{Everywhere 1-Noncatenary Local UFDs and Their Completions} \label{sec: everywhere}

We now turn our attention to everywhere 1-noncatenary local UFDs. Note that the idea of everywhere 1-noncatenary local UFDs generalizes the notion of infinitely 1-noncatenary local UFDs in that the noncatenarity condition must hold at every height 1 prime ideal instead of merely at infinitely many of them. 
In Theorem \ref{thm: main everywhere}, we present sufficient conditions on a complete local ring $T$ to ensure that it contains a local subring $A$ whose completion is $T$ and that is an everywhere 1-noncatenary UFD. Our construction of $A$ is based on the construction used in \cite{Bonat_UFDs} for the proof of Theorem \ref{thm: bonat everywhere eg}.

$N$-subrings will play a major role in our construction. We begin with two general-purpose lemmas about $N$-subrings. In the first of these, Lemma \ref{lem: adjoin transcendental}, we start with an $N$-subring that avoids a finite collection $X$ of nonmaximal prime ideals, and we show that there is a larger $N$-subring that still avoids the ideals in $X$. This lemma is different from Lemma 11 from \cite{Loepp_97} in that the collection $X$ is allowed to contain more than one element, and that we verify that the prime elements of $R$ remain prime in $S$. The proof of Lemma \ref{lem: adjoin transcendental} is the same as the proof of Lemma 11 in \cite{Loepp_97} with the exception of a few minor modifications. For the reader's convenience, we provide the proof, which in places follows verbatim the proof of Lemma 11 in \cite{Loepp_97}.


\begin{lemma} \label{lem: adjoin transcendental}
Let $(T, M)$ be a complete local ring, let $(R, M \inter R)$ be an $N$-subring of $T$, and let $X$ be a finite set of nonmaximal prime ideals of $T$ with $X_i \inter R = (0)$ for all $X_i \in X$. Suppose $C$ is a collection of nonmaximal prime ideals of $T$ with \[X \union \Ass(T) \union \rAss^{(R)}(T) \sub C.\] Let $x \in T$ be such that for every $P \in C$, $x \notin P$ and $x + P$ is transcendental over $R/(P \inter R)$ as an element of $T/P$. Then $(S, M \inter S) \coloneqq R[x]_{(M \inter R[x])}$ is an infinite $A_X$-extension of $R$.
\end{lemma}

\begin{proof}
We first show that if $P \in C$ with $P \inter R = (0)$, then $P \inter S = (0)$. It suffices to show $P \inter R[x] = (0)$. Let $f(x) = a_nx^n + \cdots + a_1x + a_0 \in P$ with $a_j \in R$ for $j = 1,2, \ldots , n$. Then $f(x) \equiv 0 \pmod{P}$. But, as $P \in C$, $x + P$ is transcendental over $R/(P \inter R)$. So, $f(x)$ is the zero polynomial (mod $P$). Hence, $a_j \in P$ for every $j = 1,2, \ldots ,n$, and we have $a_j \in P \inter R = (0)$ for every $j = 1,2, \ldots ,n$. Therefore, $f(x)$ is the zero polynomial, and we have $P \inter S = (0)$. It follows that $X_i \inter S = (0)$ for all $X_i \in X$ and $Q \inter S = (0)$ for $Q \in \Ass(T)$. Note that since $x$ is transcendental over $R$, we have that the prime elements of $R$ are prime in $S$, that $S$ is a UFD, and that $\abs{S} = \sup\set{\aleph_0, \abs{R}}$. 

To show that $S$ is an $N$-subring of $T$, it only remains to show that if $z \in T$ is regular and $P \in \Ass(T/zT)$, then $\height(P \inter S) \leq 1$. It suffices to show $\height(P \inter R[x]) \leq 1$. As $R$ is an $N$-subring, $\height(P \inter R) \leq 1$. Suppose $\height(P \inter R) = 0$. Then, since $R$ is a domain, $P \inter R = (0)$. It follows that in the ring $R[x]_{(P \inter R[x])}$, all elements of $R$ except $0$ have been inverted. So, $R[x]_{(P \inter R[x])} = k[x]$ localized at some set, where $k$ is a field. But, $\dim(k[x]) \leq 1$ and it follows that $\dim(R[x]_{(P \inter R[x])}) \leq 1$. So, $\height(P \inter R[x]) \leq 1$. Now, suppose $\height(P \inter R) = 1$. Then $P \inter R = aR$ for some $a \in P \inter R$. By Condition 2 of being an $N$-subring and the fact that $a \in R$, $a$ is not a zero divisor in $T$. Since $P \in \Ass(T/zT)$, we have $PT_P \in \Ass(T_P/zT_P)$. As $\dep(T_P) = 1$, and $a$ is regular, the ring $T_P/aT_P$ consists only of zero divisors and units. Hence, $PT_P \in \Ass(T_P/aT_P)$. It follows that $P \in \Ass(T/aT)$ and so $P \in C$. Now, if $g(x) \in P \inter R[x]$, then $g(x) = b_mx^m + \cdots + b_1x + b_0 \in P$ with $b_j \in R$. By the transcendental property of $x + P$, $b_j \in P \inter R = aR$. Hence, $g(x) \in aR[x]$, and so $\height(P \inter R[x]) = 1$. This shows that Condition 3 of $N$-subrings is satisfied and so $S$ is an $N$-subring of $T$. Thus, $S$ is an infinite $A_X$-extension of $R$.
\end{proof} 

We now present the second general-purpose result for $N$-subrings. Lemma \ref{lem: big union lemma} is very useful, as it tells us important information about unions of ascending chains of $N$-subrings. 
This result generalizes Lemma 3.6 in \cite{Loepp_Simpson} (which generalizes Lemma \ref{lem: heitmann lemma 6}) by allowing the set of prime ideals avoided by all the subrings to contain more than one element. Note that $\g(\a)$ was defined in Definition \ref{def: gamma}.

\begin{lemma} \label{lem: big union lemma}
Let $(T, M)$ be a complete local ring, let $(R_0, M \inter R_0)$ be an $N$-subring of $T$, and let $X$ be a finite subset of $\Spec(T)$ such that $X_i \inter R_0 = (0)$ for all $X_i \in X$. Let $\O$ be a well-ordered index set with least element 0 such that either $\O$ is countable, or $\abs{\set{\b \in \O \mid \b < \a}} < \abs{T/M}$ for every $\a \in \O$. Suppose further that $\set{(R_\a, M \inter R_\a)}_{\a \in \O}$ is an ascending collection of rings such that if $\g(\a) = \a$ then $R_\a = \bigunion_{\b < \a} R_\b$, while if $\g(\a) < \a$ then $R_\a$ is an $A_X$-extension of $R_{\g(\a)}$. Then $(S, M \inter S) \coloneqq \bigunion_{\a \in \O} R_\a$ satisfies all conditions to be an $N$-subring of $T$ except possibly the cardinality condition, but it satisfies $\abs{S} \leq \sup\set{\aleph_0, \abs{R_0}, \abs{\O}}$. Moreover, elements that are prime in some $R_\a$ are prime in $S$, and $X_i \inter S = (0)$ for all $X_i \in X$.
\end{lemma}

\begin{proof}
All claims follow from Lemma \ref{lem: heitmann lemma 6} except that $X_i \inter S = (0)$ for all $X_i \in X$. It suffices to show that $X_i \inter R_\a = (0)$ for all $X_i \in X$ and all $\a \in \O$. We are given that $X_i \inter R_0 = (0)$ for all $X_i \in X$. For some $\a \in \O$, suppose that for all $\b < \a$, $X_i \inter R_\b = (0)$ for all $X_i \in X$. We consider two cases. If $\g(\a) < \a$, we are done, as we have assumed that $X_i \inter R_{\g(\a)} = (0)$ for all $X_i \in X$, and we know that $R_\a$ is an $A_X$-extension of $R_{\g(\a)}$. On the other hand, if $\g(\a) = \a$, then $R_\a = \bigunion_{\b < \a} R_\b$. Using our assumption about each $R_\b$, we conclude that $X_i \inter R_\a = (0)$ for all $X_i \in X$ in this case as well.
\end{proof}



Given a complete local ring $T$ satisfying certain conditions, one of the important steps in building our everywhere 1-noncatenary UFD is to identify particular coheight one prime ideals of $T$. In Lemma \ref{lem: make the Qs}, we show that we can find these coheight one prime ideals, and we use Lemma \ref{lem: no cross-contamination} to ensure that Lemma \ref{lem: make the Qs} can be used without issue. The argument for the proof of Lemma \ref{lem: no cross-contamination} is similar to the one given in the proof of Lemma 4.5 in \cite{Bonat_UFDs}.

\begin{lemma} \label{lem: no cross-contamination}
Let $(T, M)$ be a complete local ring and suppose that $I$ and $I'$ are two distinct minimal prime ideals of $T$ such that $I + I'$ is not $M$-primary. Let $a \in T$ be such that $a \notin X_i$ for all $X_i \in \Min(I + I')$. Then no minimal prime ideal of $(a) + I$ contains $I'$.
\end{lemma}

\begin{proof}
Let $J$ be a minimal prime ideal of $(a) + I$. Suppose for contradiction that $I' \sub J$. Then $I + I' \sub J$. 
Focusing on the ring $T/I$, we have ht$(J/I) = 1$ by the principal ideal theorem. As $I/I \sub (I + I')/I \sub J/I$, we conclude that $J/I$ is a minimal prime ideal of $(I + I')/I$ in $T/I$. 
It follows that $J$ is a minimal prime ideal of $I + I'$. This is a contradiction since $a \in J$. Therefore, $I' \nsub J$, as desired.
\end{proof}

We are now ready to identify the coheight one prime ideals of $T$ that will be crucial in our construction. These prime ideals will contain only one minimal prime ideal of $T$. Furthermore, we ensure that the depth of $T$ localized at these ideals is at least two, a property that will be used in the subsequent lemma. We note that Lemma \ref{lem: make the Qs} is a generalization of Lemma 4.4 from \cite{Bonat_UFDs}, and the proof makes use of ideas in the proof of Lemma 3.6 in \cite{SMALL_2017}.

\begin{lemma} \label{lem: make the Qs}
Let $(T, M)$ be a complete local ring and let $I$ be a minimal prime ideal of $T$ satisfying the condition that dim$(T/I) \geq 3$. Let $a$ be a regular element of $T$ and suppose that the ideal $(a) + I$ has a minimal prime ideal $Q_1$ satisfying the property that $I$ is the only minimal prime ideal of $T$ contained in $Q_1$. Then there exists a prime ideal $Q$ of $T$ such that $Q_1 \subn Q$, $\dim(T/Q) = 1$, $\dep(T_Q) \geq 2$, $Q \nsub P$ for all $P \in \Ass(T)$, and $I$ is the only minimal prime ideal of $T$ contained in $Q$. 
\end{lemma}

\begin{proof}
Let $n \coloneqq \dim(T/I)$, and note that by hypothesis, $n \geq 3$. We first argue that there exists a saturated chain of prime ideals $I \subn Q_1 \subn Q_2 \subn \cdots \subn Q_{n-2}$ of $T$ such that $\dim(T/Q_{n-2}) = 2$ and $I$ is the only minimal prime ideal of $T$ contained in each $Q_i$. Then, we show that there exists a prime ideal $Q$ of $T$ such that $Q_{n-2} \subn Q \subn M$ is saturated, $\dep(T_Q) \geq 2$, $Q \nsub P$ for all $P \in \Ass(T)$, and $I$ is the only minimal prime ideal of $T$ contained in $Q$.

We create this saturated chain inductively. For the base case of $n = 3$, we consider the chain $I \subn Q_1$. In the ring $T/I$, we have ht$(Q_1/I) = 1$, so this chain is saturated. As $T$ is catenary, we have $\dim(T/Q_1) = n - 1 = 2$. By hypothesis, we also know that $I$ is the only minimal prime of $T$ contained in $Q_1$, and so we have successfully defined our saturated chain in this case.

If $n > 3$, we continue to define $Q_j$ inductively for $j \geq 2$. Suppose the prime ideals $Q_k$ of $T$ have already been defined for $k$ satisfying $1 \leq k < n - 2$. That is, $I \subn Q_1 \subn Q_2 \subn \cdots \subn Q_{k}$ is saturated, $\dim(T/Q_{k}) > 2$ and $I$ is the only minimal prime ideal of $T$ contained in each $Q_i$. 
Because $\dim(T/Q_k) > 2$ and $T$ is Noetherian, we know that there are infinitely many prime ideals $Q_{k+1}$ of $T$ such that $Q_k \subn Q_{k+1} \subn M$ and $\dim(Q_{k+1}/Q_k) = 1$. Furthermore, as $T$ is catenary, all such $Q_{k+1}$ satisfy $\dim(T/Q_{k+1}) > 1$. Suppose $I'$ is a minimal prime ideal of $T$ distinct from $I$. Then $I' \not\subseteq Q_{k}$. If $I' \sub Q_{k+1}$, then $Q_{k+1}$ is a minimal prime ideal of $Q_k + I'$. As there are only finitely many minimal prime ideals of $Q_k + I'$, we can choose $Q_{k+1}$ such that $I' \nsub Q_{k+1}$. In fact, as there are only finitely many minimal prime ideals of $T$, we can choose $Q_{k+1}$ such that $I$ is the only minimal prime ideal of $T$ it contains, and so we have successfully defined $Q_{k + 1}$. Eventually, using this inductive process, we will define $Q_{n-2}$, at which point we stop, and the construction of the saturated chain $I \subn Q_1 \subn Q_2 \subn \cdots \subn Q_{n-2}$ in $T$ is complete.

We must finally argue for the existence of our desired prime ideal $Q$ of $T$. Note that $a$ is regular in $T$ and $a \in Q_{n-2}$. By our argument above, we know that there are infinitely many prime ideals $Q$ of $T$ satisfying $Q_{n-2} \subn Q \subn M$ and $I$ is the only minimal prime ideal of $T$ contained in $Q$. We can thus choose $Q$ to additionally not be an element of $\Ass_T(T/aT)$, as this set is finite. Since $a \in Q$ is regular in $T$ and $T_Q$ is a flat extension of $T$, we know that $a$ is regular in $T_Q$. As we chose $Q$ such that $Q \notin \Ass_T(T/aT)$, we have $QT_Q \notin \Ass_{T_Q}(T_Q/aT_Q)$, and it follows that $\dep(T_Q) \geq 2$.

It remains to show that $Q \nsub P$ for all $P \in \Ass(T)$. To see this, note that $a$ is regular in $T$ and $a \in Q$. Then $Q$ cannot be contained in any associated prime ideal of $T$, as such prime ideals exclusively consist of zero divisors.
\end{proof}




 Our next goal is to use a modified version of Lemma \ref{lem: bonat 2.6} to find an extension of a given $N$-subring $R$ of $T$ that contains a generating set for each of the coheight one prime ideals obtained from Lemma \ref{lem: make the Qs}. We do this in Lemma \ref{lem: make R}, and to make sure that the conditions of that lemma are satisfied, we must check that each of our coheight one prime ideals are not contained in any $P \in \Ass(T)$ or any $P \in \rAss^{(R)}(T)$. The previous lemma verifies the former condition and the next lemma verifies the latter condition.

\begin{lemma} \label{lem: Q's are safe}
Let $(T,M)$ be a complete local ring with $\dep(T) \geq 2$, and let $Q \in \Spec(T)$ be such that $\dim(T/Q) = 1$ and $\dep(T_Q) \geq 2$. Let $z$ be a regular element of $T$, and let $P \in \Ass(T/zT)$. Then $Q \nsub P$.
\end{lemma}

\begin{proof}
As $\dim(T/Q) = 1$, it suffices to show that $P \neq M$ and $P \neq Q$. Observe that since $P \in \Ass(T/zT)$, we have $PT_P \in \Ass_{T_P}(T_P/zT_P)$, and so $\dep(T_P) = 1$. Since $\dep(T) \geq 2$ and $\dep(T_Q) \geq 2$, we have $P \neq Q$ and $P \neq M$ as desired.
\end{proof}

To simplify the proof of Lemma \ref{lem: make R}, we utilize the following preliminary lemma, in which we construct an $A_X$-extension of an $N$-subring that contains a generating set for a single one of our coheight one prime ideals. In the proof of Lemma \ref{lem: make R}, we repeatedly apply this preliminary result. Lemma \ref{lem: generating set for Q} is a generalization of Lemma 2.5 in \cite{Bonat_UFDs}, and our proof is very similar to the proof of that lemma, with relevant parts given verbatim.

\begin{lemma} \label{lem: generating set for Q}
Let $(T, M)$ be a complete local ring with $\dep(T) \geq 2$, let $(R, M \inter R)$ be an $N$-subring of $T$, and let $X$ be a finite set of nonmaximal prime ideals of $T$ such that $X_i \inter R = (0)$ for all $X_i \in X$. Let $Q$ be a prime ideal of $T$ such that $Q \nsub P$ for all $P \in X \union \Ass(T) \union \rAss^{(R)}(T)$. Then there exists an infinite $A_X$-extension $(S, M \inter S)$ of $R$ such that $S$ contains a generating set for $Q$.
\end{lemma}

\begin{proof}
Write $Q = (g_1, \ldots, g_m)$ for $g_i \in T$. Define \[C \coloneqq X \union \Ass(T) \union \rAss^{(R)}(T),\] and note that our hypotheses imply that $Q \nsub P$ for all $P \in C$. Since $\dep(T) \geq 2$, $M \notin C$. If $R$ is countable, then $C$ is countable, and if $R$ is uncountable, then $|R| < |T/M|$ and so $|C| < |T/M|$.
Use Lemma \ref{lem: heitmann lemma 2} if $C$ is countable and Lemma \ref{lem: heitmann lemma 3} if not, both with $D = \set{0}$, to conclude that $Q \nsub \bigunion_{P \in C} P$. Let $q_1 \in Q$ be such that $q_1 \notin P$ for all $P \in C$.

Fix $P \in C$. If $(g_1 + tq_1) + P = (g_1 + t'q_1) + P$ for some $t,t' \in T$, then $q_1(t - t') \in P$. Since $q_1 \notin P$, we have $t - t' \in P$ and so $t + P = t' + P$. It follows that if $t + P \neq t' + P$, then $(g_1 + tq_1) + P \neq (g_1 + t'q_1) + P$. Let $D_{(P)}$ be a full set of coset representatives for the cosets $t + P \in T/P$ that make $(g_1 + tq_1) + P$ algebraic over $R/(P \inter R)$. Note that the algebraic closure of $R/(P \inter R)$ in $T/P$ has cardinality less than or equal to $\sup\set{\aleph_0, \abs R}$. 
Let $D = \bigunion_{P \in C} D_{(P)}$. 
Use Lemma \ref{lem: heitmann lemma 2} if $R$ is countable and Lemma \ref{lem: heitmann lemma 3} if not, to find $m_1 \in M$ such that $m_1 \notin \bigunion \set{r + P \mid P \in C, r \in D}$.

Now, $(g_1 + m_1q_1) + P$ is transcendental over $R/(P \inter R)$ for all $P \in C$. Let $\tilde{g}_1 \coloneqq g_1 + m_1q_1$. By Lemma \ref{lem: adjoin transcendental}, $R_1 \coloneqq R[\tilde{g}_1]_{(M \inter R[\tilde{g}_1])}$ is an infinite $A_X$-extension of $R$. Note that $(\tilde{g}_1, g_2, \ldots, g_m) + MQ = Q$, so by Nakayama's Lemma, $Q = (\tilde{g}_1, g_2, \ldots, g_m)$.

Repeat the above process with $R$ replaced by $R_1$ to find $q_2 \in Q$ and $m_2 \in M$ so that, for $\tilde{g}_2 \coloneqq g_2 + m_2q_2$, $R_2 \coloneqq R_1[\tilde{g}_2]_{(M \inter R_1[\tilde{g}_2])}$ is an infinite $A_X$-extension of $R_1$ and $Q = (\tilde{g}_1, \tilde{g}_2, g_3, \ldots, g_m)$. Continue to find an $A_X$-extension $R_m$ of $R_{m-1}$ where $R_m$ contains a generating set for $Q$. Then $S \coloneqq R_m$ is the desired $A_X$-extension of $R$.
\end{proof}

We now have the tools to state and prove Lemma \ref{lem: make R}, an important stepping stone for our construction. In particular, given an $N$-subring $R$ of a complete local ring $T$, we show that there are sufficient conditions for there to exist an $A_X$-extension of $R$ that contains generating sets for all of our coheight one prime ideals of $T$. In the proof of Lemma \ref{lem: make R}, we construct an ascending chain of $N$-subrings of $T$ by repeatedly using Lemma \ref{lem: generating set for Q}, and we leverage Lemma \ref{lem: big union lemma} along the way. We note that Lemma \ref{lem: make R} is a generalization of Lemma 2.6 in \cite{Bonat_UFDs}.

\begin{lemma} \label{lem: make R}
Let $(T, M)$ be a complete local ring with $\dep(T) \geq 2$, let $(R_0, M \inter R_0)$ be an $N$-subring of $T$, and let $X$ be a finite set of nonmaximal prime ideals of $T$ such that $X_i \inter R_0 = (0)$ for all $X_i \in X$. Let $Y \sub \Spec(T)$ be such that $\abs{Y} \leq \sup\set{\aleph_0, \abs{T/M}}$, with equality only if $T/M$ is countable, and suppose that for all $Q \in Y$ we have $Q \nsub P$ for all $P \in X \union \Ass(T) \union \rAss^{(R_0)}(T)$. Then there exists an $N$-subring $(S, M \inter S)$ of $T$ such that $S$ contains a generating set for every $Q \in Y$. Moreover, $S$ satisfies all the conditions to be an $A_X$-extension of $R_0$ except possibly the cardinality condition, but it satisfies $\abs{S} \leq \sup\set{\aleph_0, \abs{R_0}, \abs{Y}}$.
\end{lemma}

\begin{proof}
Write $Y = \set{Q_\ell}_{\ell \in L'}$, where $L'$ is some well-ordered index set with minimal element 0 that is countable or satisfies $\abs{\set{k \in L' \mid k < \ell}} < \abs{T/M}$ for all $\ell \in L'$. If $L'$ does not have a maximal element, let $L = L'$, and if $L'$ does have a maximal element, define $L = L' \cup \{\alpha\}$ where $\alpha > \ell$ for all $\ell \in L'$. Let $R_1$ be the infinite $A_X$-extension of $R_0$ obtained by applying Lemma \ref{lem: generating set for Q} with $Q = Q_0$. We recursively define $R_\ell$ for all other $\ell \in L$. Recall that we set $\g(\ell) \coloneqq \sup\set{k \in L \mid k < \ell}$. To define $R_\ell$, we proceed in one of two ways. If $\g(\ell) < \ell$, let $R_\ell$ be the $A_X$-extension of $R_{\g(\ell)}$ obtained by applying Lemma \ref{lem: generating set for Q} with $Q = Q_{\g(\ell)}$. Then $R_\ell$ contains a generating set for every $Q_k$ satisfying $k < \ell$. Alternatively, if $\g(\ell) = \ell$, simply set $R_\ell = \bigunion_{k < \ell} R_k$, which ensures that $R_\ell$ contains a generating set for every $Q_k$ satisfying $k < \ell$. 
Then, by Lemma \ref{lem: big union lemma}, $S \coloneqq \bigunion_{\ell \in L} R_\ell$ is the desired $A_X$-extension of $R_0$. 
\end{proof}


Recall that given a complete local ring $T$, we want to construct a subring $A$ of $T$ whose completion is $T$. To show this, we use Theorem \ref{thm: completion-proving machine}. Thus, we need the natural map $A \to T/M^2$ to be surjective and $IT \inter A = I$ for every finitely generated ideal $I$ of $A$. 
We use the next lemma to construct $A$ so that the first of these properties is satisfied. It is a direct generalization of Lemma 3.5 in \cite{Loepp_Simpson}, in that we now allow $X$ to contain more than one element and $R$ to be finite. We reproduce that proof here with small changes.



\begin{lemma} \label{lem: put t in image}
Let $(T, M)$ be a complete local ring with $\dep(T) \geq 2$, let $(R, M \inter R)$ be an $N$-subring of $T$, and let $X$ be a finite set of nonmaximal prime ideals of $T$ such that $X_i \inter R = (0)$ for all $X_i \in X$. Let $t \in T$. Then there exists an infinite $A_X$-extension $(S, M \inter S)$ of $R$ such that $t + M^2$ is in the image of the map $S \to T/M^2$.
\end{lemma}

\begin{proof}
Define \[C \coloneqq X \union \Ass(T) \union \rAss^{(R)}(T).\] Since $\dep(T) \geq 2$, $M \notin C$. Let $P \in C$ and let $D_{(P)}$ be a full set of coset representatives of the elements $z + P \in T/P$ that make $z + t + P$ algebraic over $R$. Define $D = \bigunion_{P \in C} D_{(P)}$. Since $M \nsub P$ for all $P \in C$, we have $M^2 \nsub P$ for all $P \in C$. Use Lemma \ref{lem: heitmann lemma 2} if $R$ is countable and Lemma \ref{lem: heitmann lemma 3} if not, both with $I = M^2$, to find $x \in M^2$ such that $x \notin \bigunion \set{r + P \mid P \in C, r \in D}$. Define $S \coloneqq R[x + t]_{(M \inter R[x + t])}$. Note that $t + M^2$ is in the image of the map $S \to T/M^2$. By Lemma \ref{lem: adjoin transcendental}, $S$ is an infinite $A_X$-extension of $R$.
\end{proof}



We use the next lemma to show that the second condition needed to apply Theorem \ref{thm: completion-proving machine} holds.  That is, for the subring $A$ of our complete local ring $T$, we have $IT \cap A = I$ for every finitely generated ideal $I$ of $A$.
%
The lemma differs from Lemma 3.7 in \cite{Loepp_Simpson} by allowing $X$ to contain more than one element. Thus, the proof of Lemma \ref{lem: make A} is very similar to the proof of Lemma 3.7 in \cite{Loepp_Simpson}, with some parts taken exactly as stated in that proof.

\begin{lemma} \label{lem: make A}
Let $(T, M)$ be a complete local ring with $\dep(T) \geq 2$, let $(R, M \inter R)$ be an $N$-subring of $T$, and let $X$ be a finite set of nonmaximal prime ideals of $T$ such that $X_i \inter R = (0)$ for all $X_i \in X$. Let $t \in T$. Then there exists an infinite $A_X$-extension $(S, M \inter S)$ of $R$ such that $t + M^2$ is in the image of the map $S \to T/M^2$, and for every finitely generated ideal $I$ of $S$ we have $IT \inter S = I$.
\end{lemma}

\begin{proof}
Apply Lemma \ref{lem: put t in image} to find an infinite $A_X$-extension $R_0$ of $R$ such that $t + M^2$ is in the image of the map $R_0 \to T/M^2$. We construct an $A_X$-extension $S$ of $R_0$ such that $IT \inter S = I$ for every finitely generated ideal $I$ of $S$.

Let \[\O \coloneqq \set{(I,c) \mid I \text{ is a finitely generated ideal of } R_0 \text{ and } c \in IT \inter R_0},\] and note that $\abs{\O} = \abs{R_0}$. Well-order $\O$ so that it has minimal element 0 and so that it does not have a maximal element.  
We inductively define an ascending collection $\set{R_\a}_{\a \in \O}$ of $N$-subrings of $T$. We have already defined $R_0$. Let $\a \in \O$ and assume that $R_\b$ has been defined for all $\b < \a$. 
If $\g(\a) < \a$, write $\g(\a) = (I,c)$, and define $R_\a$ to be the $A_X$-extension of $R_{\g(\a)}$ obtained from Lemma \ref{lem: jensen} (taking $R = R_{\g(\a)}$ and $a = 0$) so that $\abs{R_\a} = \abs{R_{\g(\a)}}$, and $c \in IR_\a$. 
If $\g(a) = \a$, define $R_\a = \bigunion_{\b < \a} R_\b$. 

With our ascending collection of $N$-subrings defined, set $R_1 \coloneqq \bigunion_{\a \in \O} R_\a$. 
By Lemma \ref{lem: big union lemma}, $R_1$ is an $A_X$-extension of $R_0$ and $|R_1| = |R_0|$. For $I$ a finitely generated ideal of $R_0$ and $c \in IT \inter R_0$, we know $(I,c) = \g(\a)$ for some $\a \in \O$ with $\g(\a) < \a$. By construction, $c \in IR_\a \sub IR_1$. It follows that $IT \inter R_0 \sub IR_1$ for every finitely generated ideal $I$ of $R_0$.

Repeat the above construction replacing $R_0$ with $R_1$ to find an $A_X$-extension $R_2$ of $R_1$ such that $|R_2| = |R_0|$ and $IT \inter R_1 \sub IR_2$ for every finitely generated ideal $I \sub R_1$. Continue to obtain an infinite chain of $N$-subrings $R_1 \sub R_2 \sub R_3 \sub 
\cdots$ of $T$ such that, for all integers $n \geq 1$, $R_{n+1}$ is an $A_X$-extension of $R_n$, $|R_n| = |R_0|$, and $IT \inter R_n \sub IR_{n+1}$ for every finitely generated ideal $I$ of $R_n$.

Let $S \coloneqq \bigunion_{i=1}^\infty R_i$. By Lemma \ref{lem: big union lemma}, $S$ is an $A_X$-extension of $R_1$. Since $S$ is an $A_X$-extension of $R_1$ and $R_1$ is an $A_X$-extension of $R_0 = R$, we have that $S$ is an $A_X$-extension of $R$. Let $I = (g_1, \ldots, g_m)$ for some $g_i \in S$ be a finitely generated ideal of $S$. Let $c \in IT \inter S$. Then there is an $r \in \N$ such that $c \in R_r$ and $g_i \in R_r$ for all $i \in \set{1, 2, \ldots, m}$. Now, $c \in (g_1, \ldots, g_m)T \inter R_r \sub (g_1, \ldots, g_m)R_{r+1} \sub I$, and so $IT \inter S = I$. Note also that as $t + M^2$ was in the image of the map $R_0 \to T/M^2$, it will still be in the image of the map $S \to T/M^2$, and so it follows that $S$ is the desired $N$-subring of $T$.
\end{proof}

We now have all the tools in place to prove the second main theorem of the paper, in which we find sufficient conditions for a complete local ring $(T,M)$ to be the completion of an everywhere 1-noncatenary local UFD $A$. For the proof, we want to construct an $N$-subring $A$ of $T$ that satisfies the hypotheses of Theorem \ref{thm: completion-proving machine}, as well as the hypotheses of Lemma \ref{lem: noncatenary-proving machine} with respect to each of its height one prime ideals. 
It does not suffice to simply construct an $N$-subring that satisfies the hypotheses of one of the results and then construct a larger $N$-subring that satisfies the hypotheses of the other since that $N$-subring may no longer satisfy the hypotheses of the first result.
Thus, we alternate between applying Lemma \ref{lem: make A} and applying Lemma \ref{lem: make R}, and show that our desired subring $A$ is the union of the subrings obtained by this process.

\begin{theorem} \label{thm: main everywhere}
Let $(T, M)$ be a complete local ring, with $\Min(T) = \set{P^{(1)}, P^{(2)}, \ldots, P^{(m)}}$, satisfying the following conditions:
\begin{enumerate}
    \item $\dep(T) \geq 2$,
    \item $\dim(T/P^{(1)}) \geq 3$, $\dim(T/P^{(2)}) \geq 3$, and $\dim(T/P^{(1)}) \neq \dim(T/P^{(2)})$, \label{con: two minimals}
    \item $T$ contains a field, and 
    \label{con: field}
    \item for $1 \leq i \leq 2$ and $1 \leq j \leq m$, $P^{(i)} + P^{(j)}$ is not $M$-primary.
    \label{con: nonmaximal}
\end{enumerate}
Then $T$ is the completion of an everywhere 1-noncatenary local UFD $(A, M \inter A)$.
\end{theorem}

\begin{proof}
If $T/M^2$ is infinite, then let $\set{t_\a}_{\a \in \O}$ be a complete set of coset representatives for the elements of $T/M^2$, where $\O$ is a well-ordered index set such that every element of $\O$ has fewer than $|\O|$ predecessors. Let 0 denote the minimal element of $\O$. If $T/M^2$ is finite, then let $\O$ be the set of nonnegative integers, and let $\{t_0, t_1, \ldots ,t_n\}$ be a full set of coset representatives for the elements of $T/M^2$. For $\ell > n$, define $t_{\ell} = t_0$.
We construct a family of rings $\set{A_\a}_{\a \in \O}$ that satisfy the hypotheses of Lemma \ref{lem: heitmann lemma 6}, the union of which will yield our desired everywhere 1-noncatenary local UFD $A$. 

Let $\Pi$ denote the prime subring of $T$, and set $A_0 \coloneqq \Pi_{(M \inter \Pi)}$. Then $A_0$ is a field and a countable $N$-subring of $T$. 
Define \[X \coloneqq \bigunion_{\substack{1 \leq i \leq 2 \\ 1 \leq j \leq m}} \Min(P^{(i)} + P^{(j)}).\] Then $X_i \inter A_0 = (0)$ for all $X_i \in X$.


We now define our family of rings $\set{A_\a}_{\a \in \O}$ recursively. To create $A_\a$ for $\a \in \O$, we assume that $\set{A_\b}_{\b < \a}$ is already defined such that, for $\beta < \alpha$, we have that $A_{\beta}$ is an $N$-subring of $T$ and $X_i \cap A_{\beta} = (0)$ for all $X_i \in X$. We first find $A_\a$ in the case $\g(\a) < \a$. Before constructing $A_{\a}$ in this case, we find an $N$-subring $R_\a$ of $T$ that we use to define $A_{\a}$.

As $A_{\g(\a)}$ is an $N$-subring of $T$, it is necessarily a UFD. Thus, its height one prime ideals are principal. Let $\set{a_k}_{k \in K}$ be a set of generators of all height one prime ideals of $A_{\g(\a)}$ for some index set $K$ with minimal element 0. 
Note that as $X_i \inter A_{\g(\a)} = (0)$ for every $X_i \in X$, $a_k \notin X_i$ for all $X_i \in X$ and for all $k \in K$. Note also that as each $a_k$ is a nonzero element of the $N$-subring $A_{\g(\a)}$, we can use condition 2 of the definition of $N$-subring to see that each $a_k$ is contained in no associated prime ideal of $T$, and in particular in no minimal prime ideal of $T$.

Now, let $P^{(1)}_0 \in \Spec(T)$ be a minimal prime ideal of $(a_0) + P^{(1)}$ and let $P^{(2)}_0 \in \Spec(T)$ be a minimal prime ideal of $(a_0) + P^{(2)}$. By Lemma \ref{lem: no cross-contamination} with $a = a_0$, $I = P^{(1)}$, and $I'$ any other minimal prime ideal of $T$, we find that $P^{(1)}$ is the only minimal prime ideal of $T$ contained in $P^{(1)}_0$. Similarly, $P^{(2)}$ is the only minimal prime ideal of $T$ contained in $P^{(2)}_0$.
By Lemma \ref{lem: make the Qs} with $a = a_0$, $I = P^{(1)}$, and $Q_1 = P^{(1)}_0$, there is a coheight one prime ideal $Q^{(1)}_0$ of $T$ containing $P^{(1)}_0$ such that $\dep(T_{Q^{(1)}_0}) \geq 2$, $Q^{(1)}_0 \nsub P$ for all $P \in \Ass(T)$ and $P^{(1)}$ is the only minimal prime ideal of $T$ contained in $Q^{(1)}_0$. Similarly, a prime ideal $Q^{(2)}_0$ of $T$ exists with analogous properties. Repeat this procedure to find prime ideals $Q^{(1)}_k$ and $Q^{(2)}_k$ of $T$ for each $a_k$, and set $Y_{\g(\a)} \coloneqq \bigunion_{k \in K} \{Q^{(1)}_k, Q^{(2)}_k\}$.

 By construction, $Q^{(1)}_k$ and $Q^{(2)}_k$ are not contained in any element of $\Ass(T)$ for all $k \in K$. By Lemma \ref{lem: Q's are safe}, $Q^{(1)}_k$ and $Q^{(2)}_k$ are not contained in any element of $\rAss(T)$ for every $k \in K$. We can thus apply Lemma \ref{lem: make R} with $R = A_{\g(\a)}$ and $Y = Y_{\g(\a)}$ to find an $A_X$-extension $R_\a$ of $A_{\gamma(\alpha)}$ such that, for all $k \in K$, $R_\a$ contains a generating set for $Q^{(1)}_k$ and $Q^{(2)}_k$. 

Apply Lemma \ref{lem: make A} with $R = R_\a$ and $t = t_{\g(\a)}$ to get an $A_X$-extension $A_\a$ of $R_\a$ such that $t_{\g(\a)} + M^2$ is in the image of the map $A_\a \to T/M^2$, and for every finitely generated ideal $I$ of $A_\a$ we have $IT \inter A_\a = I$. As $A_\a$ is an $A_X$-extension of $R_\a$ and $R_\a$ is an $A_X$-extension of $A_{\g(\a)}$, we have that $A_\a$ is an $A_X$-extension of $A_{\g(\a)}$.

In the case where $\g(\a) = \a$, let $A_\a = \bigunion_{\b < \a} A_\b$.
By Lemma \ref{lem: big union lemma}, $A_\a$ is an $N$-subring of $T$ such that prime elements of any $A_\b$ are prime in $A_\a$ and $X_i \inter A_\a = (0)$ for all $X_i \in X$.

With our family of rings defined, let $A \coloneqq \bigunion_{\a \in \O} A_\a$. By Lemma \ref{lem: heitmann lemma 6}, $A$ satisfies all conditions of being an $N$-subring of $T$ other than the cardinality condition. Furthermore, a prime element in any subring of this union is prime in $A$. Note that $A$ is a UFD. We claim that $IT \inter A = I$ for all finitely generated ideals $I$ of $A$. Let $I = (g_1, \ldots, g_t)$ be a finitely generated ideal of $A$ and let $c \in IT \inter A$. Then we can choose $N \in \O$ with $\gamma(N) < N$ such that $\set{c, g_1, \ldots, g_t} \sub A_N$. Now, \[c \in (g_1, \ldots, g_t)T \inter A_N = (g_1, \ldots, g_t)A_N \sub I,\] and so $IT \inter A = I$ as claimed.
Furthermore, by construction, the map $A \to T/M^2$ is onto. 
We conclude by Proposition \ref{thm: completion-proving machine} that $A$ is Noetherian and $\widehat{A} \iso T$.

Let $J$ be a height one prime ideal of $A$. Since $A$ is a UFD, $J$ is principal, so we can write $J = aA$ for some prime element $a \in A$. 
Choose $N' \in \O$ so that $\gamma(N') < N'$ and $a \in A_{N'}$. We show that $a$ is prime in $A_{N'}$. Suppose $a = p_1 p_2 \cdots p_m$ is the prime factorization of $a$ in $A_{N'}$. Prime elements in $A_{N'}$ are prime in $A$, so $a = p_1 p_2 \cdots p_m$ is also the prime factorization of $a$ in $A$. This implies that $m = 1$, and so $a$ is prime in $A_{N'}$ as desired.

By construction, $T$ contains two coheight one prime ideals $Q^{(1)}_a$ and $Q^{(2)}_a$, where $Q^{(1)}_a$ contains a minimal prime ideal $P^{(1)}_a$ of $(a) + P^{(1)}$ and $Q^{(2)}_a$ contains a minimal prime ideal $P^{(2)}_a$ of $(a) + P^{(2)}$. By the principal ideal theorem, ht$(P^{(1)}_a/P^{(1)}) = 1$ and ht$(P^{(2)}_a/P^{(2)}) = 1$. Furthermore, $P^{(1)}$ is the only minimal prime of $T$ contained in $Q^{(1)}_a$, and $P^{(2)}$ is the only minimal prime of $T$ contained in $Q^{(2)}_a$. Also, by construction, $A_{N'+1}$ (and therefore $A$) contains a generating set for $Q^{(1)}_a$ and $Q^{(2)}_a$. As $P^{(1)}_a$ and $P^{(2)}_a$ are both height one prime ideals of $T$, $\height(P^{(1)}_a \inter A) = 1$ and $\height(P^{(2)}_a \inter A) = 1$. As both contain the prime element $a$ of $A$, we find that $P^{(1)}_a \inter A = aA$ and $P^{(2)}_a \inter A = aA$.

To complete the proof, apply Lemma \ref{lem: noncatenary-proving machine} with $Q^{(i)} = Q^{(i)}_a$ and $P^{(i)}_1 = P^{(i)}_a$ for $i \in \set{1,2}$. This confirms that $A/J$ is noncatenary, and therefore $A$ is everywhere 1-noncatenary.
\end{proof}

We end with an example illustrating that, for $n \geq 4$, there exists an everywhere 1-noncatenary local UFD of dimension $n$.

\begin{example}\label{example}
Let $n \geq 4$ and let $T = K[[x_1, x_2, \ldots ,x_n, x_{n + 1}]]/((x_1) \cap (x_2,x_3))$ where $K$ is a field. Then dim$(T) = n$, and $T$ satisfies the hypotheses of Theorem \ref{thm: main everywhere}. Thus, $T$ is the completion of an everywhere 1-noncatenary local UFD $A$. It follows that dim$(A) = n$.
\end{example}


\end{document}